\documentclass[letterpaper]{article}

\title{The Lefthanded Local Lemma characterizes chordal dependency graphs}
\author{Wesley Pegden\footnote{
Courant Institute of Mathematical Sciences, New York University,
251 Mercer St, Rm 921,
New York, NY 10012
Email: pegden@math.nyu.edu.  Partially supported by NSF MSPRF grant 1004696.}
}
\date{March 30, 2012}

\usepackage{amsmath,amsthm,amssymb,amscd,latexsym,pstricks,pst-node,pst-tree}
\usepackage{subfigure}
\usepackage{url}
\usepackage{cite}

\newcommand{\nin}{\notin}

\newcommand{\pp}{{\rm P}}

\newcommand{\sss}{{\cal S}}
\newcommand{\aaa}{{\cal A}}

\newcommand{\BB}{{\cal B}}

\newcommand{\LLL}{{\cal L}}

\newcommand{\empt}{\varnothing}

\newcommand{\dun}{\;\dot\cup\;}

\newcommand{\sbs}{\subset}
\newcommand{\abs}[1]{\lvert #1 \rvert}

\newcommand{\st}{\,\vrule\,}

\newcommand{\bfrac}[2]{\frac{\displaystyle #1}{\displaystyle #2}}

\newtheorem{theorem}{Theorem}[section]

\newtheorem{cor}[theorem]{Corollary}

\newtheorem{observ}[theorem]{Observation}

\newcommand{\comments}[1]{}

{
\theoremstyle{definition}
\newtheorem{definition}[theorem]{Definition}

}

{
\theoremstyle{remark}

\newtheorem{ex}{Example}
}

\newcommand{\stm}{\setminus}

\newcommand{\ep}{\varepsilon}

\newcommand{\VX}[2]{\cnode*(#1){2.5pt}{#2}}

\newcommand{\edg}[2]{\ncline{#1}{#2}}
\newcommand{\oedg}[2]{\ncline[linestyle=dashed, dash=2pt 1.5pt, arrows=->,arrowsize=4pt 4]{#1}{#2}}

\newcommand{\Cedg}[2]{\nccurve[arcangleA=50,arcangleB=50,ncurv=2]{#1}{#2}}

\SpecialCoor
\begin{document}
\maketitle
\begin{abstract}
Shearer gave a general theorem characterizing the family $\LLL$ of dependency graphs labeled with probabilities $p_v$ which have the property that for any family of events with a dependency graph from $\LLL$ (whose vertex-labels are upper bounds on the probabilities of the events), there is a positive probability that none of the events from the family occur.

We show that, unlike the standard Lov\'asz Local Lemma---which is less powerful than Shearer's condition on every nonempty graph---a recently proved `Lefthanded' version of the Local Lemma is equivalent to Shearer's condition for all chordal graphs.  This also leads to a simple and efficient algorithm to check whether a given labeled chordal graph is in $\LLL$.
\end{abstract}

\section{Introduction}
Suppose we would like to show that, with positive probability, none of the events from a finite family $A_1,\dots, A_n$ occur.  If the events are independent, then the probability that none occur is $\prod(1-\pp(A_i))$ and this conclusion would follow from requiring simply that $\pp(A_i)<1$ for all $1\leq i\leq n$.  The Lov\'asz Local Lemma, proved by Erd\H{o}s and Lov\'asz in 1975\cite{EL}, allows the same kind conclusion when the independence condition is relaxed in exchange for stronger bounds on the $\pp(A_i).$  The relaxation of the independence condition is captured by a \emph{dependency graph}:
\begin{definition}
A graph $G$ with finite vertex set $V$ is a \emph{dependency graph} for a family $\{A_v\}_{v\in V}$ of events if each $A_v$ is independent of any family of events whose corresponding vertices are all nonadjacent to $v$ in $G$.
\end{definition}
The general version of the Lov\'asz Local Lemma can be stated as follows:
\begin{theorem}[Lov\'asz Local Lemma]
    Consider a family of events $\{A_v\}_{(v\in V)}$ with a dependency graph $G$ on $V$.
If there is an assignment of numbers $0\leq x_v\leq 1$ ($v\in V$) such that for any $v\in G$ 
  \begin{equation}
    \pp\left(A_v\right)\leq x_v\prod_{u\sim v}(1-x_u),\textrm{ then}
    \label{l.lllcond}
  \end{equation}
  \begin{equation}
    \pp\left(\bigcap_{A\in \aaa}\bar A\right)\geq \prod_{v\in G}(1-x_v),
  \end{equation}
and so in particular, if $x_v<1$ for all $v$ then we have that
  \label{t.lll}
  \begin{equation}
    \pp\left(\bigcap_{A\in \aaa}\bar A\right)>0.
    \label{l.conc}
  \end{equation}
\end{theorem}
Not only does the Local Lemma allow one to conclude that $\pp\left(\bigcap_{A\in \aaa}\bar A\right)>0$ without having full-blown independence of the $A_i$, it does so with only `local' conditions; that is, each instance of condition (\ref{l.lllcond}) concerns only a single vertex and its neighborhood.  For example, if $G$ has maximum degree $\Delta$, then making the assignment $x_v=\frac{1}{\Delta+1}$ for all $v$ gives the conclusion of the Lemma under the condition that for all $v$ we have $\pp(A_v)\leq \frac{\Delta^\Delta}{(\Delta+1)^{\Delta+1}}$, and for this condition to be satisfied it is enough to have $\pp(A_v)\leq \frac{1}{e\Delta}$.  The fact that the Lemma can be applied with only local conditions on the dependency graph means that it can be applied without detailed knowledge of the structure of the dependency graph: knowledge about the `local size' is sufficient.  This has allowed the Local Lemma to become a central tool in probabilistic combinatorics, used to prove the existence of combinatorial objects with wide ranges of properties.

We define $\LLL$ as the family of graphs $G$ with vertices labeled with real numbers $0\leq p_v\leq 1$ with the property that for any family of events $A_v$ having $G$ as a dependency graph and for which $\pp(A_v)\leq p_v$ for all $v\in V(G)$, the family of events satisfies $\pp(\bigcap \bar A_v)>0$.  Separate from the question of when we can fruitfully apply the Local Lemma to a combinatorial problem, there is a natural theoretical question regarding which labeled graphs are in $\LLL$.  Since the Local Lemma only uses `local' conditions on the graph, it is not surprising that there should be some labeled graphs in $\LLL$ to which the Local Lemma doesn't apply.  For example, if $K_2^p$ is a graph with 2 vertices joined by an edge, both labeled with the same probability $p$, then the conditions of the Local Lemma apply exactly if $p\leq \frac 1 4$, even though $K_2^p\in \LLL$ for all $p<\frac 1 2$.

In \cite{shear}, Shearer gave a complete characterization of the family $\LLL$:
\begin{theorem}[Shearer]
\label{t.S}
  Let $G$ be graph labeled with numbers $0\leq p_v\leq 1$ ($v\in G$).  For $S\sbs V(G)$ let
\begin{equation}
\sss(S):=\sum_{\substack{I\supset S\\I\textrm{ indep.}}}(-1)^{\abs{I}-\abs{S}}\prod_{v\in I}p_v.
\label{l.sdef}
\end{equation}
If $\sss(S)\geq 0$ for all $S$, then 
\[
\pp\left(\bigcap_{v\in G}\bar A\right)\geq \sss(\empt)
\]
for any set of events $A_v$ with $\pp(A_v)\leq p_v$ for all $v$ and for which $G$ is a dependency graph (with the vertex $v$ corresponding to the event $A_v$).  Furthermore, this bound is best possible, and if there is any $S'$ with $\sss(S')<0$, then there is a family of events $A_v$ with $\pp(A_v)\leq p_v$ for all $v$ and for which $G$ is a dependency graph such that $\pp\left(\bigcap\limits_{v\in G}\bar A\right)=0$.
\end{theorem}
\begin{ex}
  If $K_n^p$ is the complete graph in which every vertex gets the same label $p\in [0,1]$, then the conclusion (\ref{l.conc}) of the Local Lemma applies to any family of events with $K_n^p$ as their dependency graph if and only if $p<\frac{1}{n}$.  Accordingly, we see that for $K_n^p$, the only time the sum (\ref{l.sdef}) may be negative is when $S=\varnothing$, and that we have $\sss(\varnothing)>0$ if and only if $1-np>0$, and so whenever $p<\frac 1 n$.
\end{ex}

In spite of its theoretical importance, Shearer's condition is typically computationally intractable in practice.  (Apart from likely difficulties in finding independent sets in the graph, note that the sums (\ref{l.sdef}) may contain an exponential number of terms.)  In fact, it remains unclear whether or not the problem of deciding whether a given labeled graph is in $\LLL$ is in NP.  Nevertheless, we will see that for a restricted class of dependency graphs (which arises combinatorially in the consideration of problems on sequences, for example) determining membership in $\LLL$ is `easy', and, surprisingly, can be done with only local conditions on the dependency graph.  First we need some definitions.
\begin{definition}
A \emph{tree order} is a partial order in which $w\lneq u,v$ implies that $u$ and $v$ are comparable. 
\end{definition}
\noindent (In particular, a linear order is a tree order.)
\begin{definition}
\label{d.left}
A graph is a \emph{lefthanded graph} with respect to a tree-order $\leq$ if
\begin{enumerate}
\item $u\sim v$ implies that $u\leq v$ or $v\leq u$, and \label{c.adjwintree}
\item $(w\lneq u\lneq v)$ and $(v\sim w)$ together imply $(v\sim u)$. \label{c.lordc}
\end{enumerate}
\end{definition}

\smallskip
A \emph{subtree graph} is a graph whose vertex-set is a set of subtrees of some fixed tree, where adjacency corresponds to intersection.  The following simple observation gives us an `order-free' view of the lefthanded graphs:
\begin{observ}
  A graph is lefthanded with respect to some tree-order if and only if it is isomorphic to a subtree graph.
\label{o.leftissub}
\end{observ}
\noindent A proof is given at the beginning of Section \ref{s.proof}.  In light of Observation \ref{o.leftissub}, and recalling that a graph is \emph{chordal} if it has no cycles of length $>3$ as induced subgraphs, the following theorem of Gavril nicely characterizes the lefthanded graphs.
\begin{theorem}[Gavril (1974)]
A graph is a subtree graph if and only if it is chordal.  Moreover, a subtree graph isomorphic to a given chordal graph can be found in polynomial time.
\label{t.G}
\end{theorem}
\noindent The second part of Theorem \ref{t.G} implies that we can efficiently find a suitable tree-order for any chordal graph:
\begin{cor}
For any chordal graph $G$, a tree order $\leq$ such that $(G,\leq)$ is a lefthanded graph can be found in polynomial time.
\label{c.G}
\end{cor}
\noindent Figure \ref{f.GH} shows the Goldner--Harary graph (famous for being maximally planar yet non-Hamiltonian) which is chordal, together with the tree-order which realizes it as a lefthanded graph.   (The relationship between subtree graphs and lefthanded graphs giving rise to the corollary is made explicit in the proof of Observation \ref{o.leftissub} given at the beginning of Section \ref{s.proof}.)

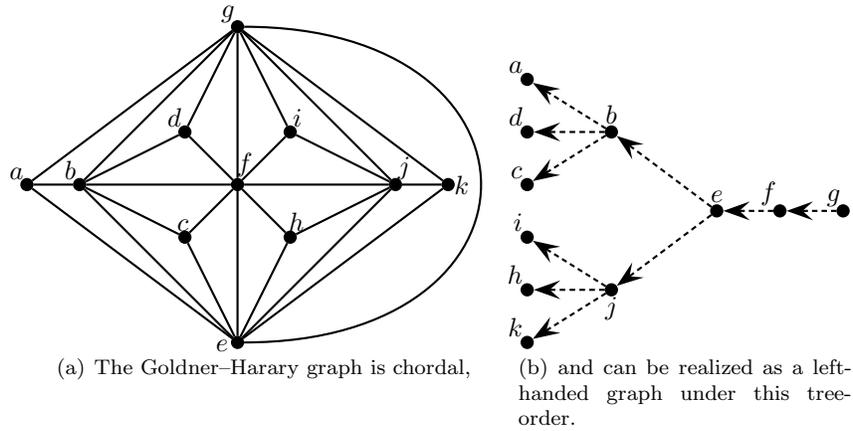
\begin{figure}
\begin{center}
\psset{unit=.7cm}
\subfigure[The Goldner--Harary graph is chordal,]{
\begin{pspicture}(1,1)(10,7)
\VX{1,4}{Va}
\VX{2,4}{Vb}
\VX{4,3}{Vc}
\VX{4,5}{Vd}
\VX{5,1}{Ve}
\VX{5,4}{Vf}
\VX{5,7}{Vg}
\VX{6,3}{Vh}
\VX{6,5}{Vi}
\VX{8,4}{Vj}
\VX{9,4}{Vk}

\uput{.12}[130](Va){$a$}
\uput{.12}[130](Vb){$b$}
\uput{.12}[100](Vc){$c$}
\uput{.12}[130](Vd){$d$}
\uput{.18}[190](Ve){$e$}
\uput{.15}[65](Vf){$f$}
\uput{.10}[130](Vg){$g$}
\uput{.12}[65](Vh){$h$}
\uput{.12}[55](Vi){$i$}
\uput{.12}[55](Vj){$j$}
\uput{.12}[0](Vk){$k$}

\edg{Va}{Vb}
\edg{Va}{Ve}
\edg{Va}{Vg}

\edg{Vb}{Vc}
\edg{Vb}{Vd}
\edg{Vb}{Ve}
\edg{Vb}{Vf}
\edg{Vb}{Vg}

\edg{Vd}{Vg}
\edg{Vd}{Vf}

\edg{Vc}{Ve}
\edg{Vc}{Vf}

\edg{Ve}{Vf}
\edg{Vf}{Vg}

\edg{Vk}{Vj}
\edg{Vk}{Ve}
\edg{Vk}{Vg}

\edg{Vj}{Vh}
\edg{Vj}{Vi}
\edg{Vj}{Ve}
\edg{Vj}{Vf}
\edg{Vj}{Vg}

\edg{Vi}{Vg}
\edg{Vi}{Vf}

\edg{Vh}{Ve}
\edg{Vh}{Vf}

\Cedg{Ve}{Vg}
\end{pspicture}
}
\subfigure[and can be realized as a lefthanded graph under this tree-order.]{
\begin{pspicture}(1,1)(7,6)
\VX{1,1}{Vk}
\VX{1,2}{Vh}
\VX{1,3}{Vi}
\VX{1,4}{Vc}
\VX{1,5}{Vd}
\VX{1,6}{Va}

\VX{2.6,2}{Vj}
\VX{2.6,5}{Vb}

\VX{4.6,3.5}{Ve}
\VX{5.8,3.5}{Vf}
\VX{7,3.5}{Vg}

\uput{.15}[130](Va){$a$}
\uput{.16}[90](Vb){$b$}
\uput{.16}[130](Vc){$c$}
\uput{.15}[130](Vd){$d$}
\uput{.16}[90](Ve){$e$}
\uput{.13}[130](Vf){$f$}
\uput{.10}[130](Vg){$g$}
\uput{.15}[130](Vh){$h$}
\uput{.15}[130](Vi){$i$}
\uput{.14}[-90](Vj){$j$}
\uput{.15}[130](Vk){$k$}

\oedg{Vg}{Vf}
\oedg{Vf}{Ve}

\oedg{Ve}{Vb}
\oedg{Vb}{Va}
\oedg{Vb}{Vd}
\oedg{Vb}{Vc}

\oedg{Ve}{Vj}
\oedg{Vj}{Vk}
\oedg{Vj}{Vh}
\oedg{Vj}{Vi}

\end{pspicture}
}
\end{center}

\vspace{-.5cm}
  \caption{By Gavril, every chordal graph is a lefthanded graph with respect to some (not necessarily unique) tree-order.\label{f.GH}}
\end{figure}
\bigskip

In \cite{leftseq}, we proved a `Lefthanded Local Lemma', for the purpose of proving the existence of winning strategies in sequence games.  Roughly speaking, it allows one to ignore dependencies `to the right' for lefthanded dependency graphs.  For undirected graphs, the Lefthanded Local Lemma can be stated as follows:
\begin{theorem}[Lefthanded Local Lemma]
  Consider a family of events $\{A_v\}_{(v\in V)}$ with a lefthanded dependency graph ($G$,$\leq$) on $V$.
If there is an assignment of numbers $0\leq x_v\leq 1$ ($v\in V$) such that for any $v\in G$ 
  \begin{equation}
    \pp\left(A_v\right)\leq x_v\prod_{\substack{u\sim v\\ u\leq v}}(1-x_u).
    \label{l.cond}
  \end{equation}
Then we have
  \begin{equation}
    \pp\left(\bigcap_{A\in \aaa}\bar A\right)\geq \prod_{v\in G}(1-x_v)
\label{l.lconc}
  \end{equation}
  \label{t.ol}
\end{theorem}
\noindent  The version of Theorem \ref{t.ol} proved in \cite{leftseq} is, strictly speaking, more general than we have stated it here, applying to some dependency graphs which are not lefthanded in the sense of Definition \ref{d.left}.

In this note, we prove the following:
\begin{theorem}
\label{t.lisbest}
  If $(G,\leq)$ is a lefthanded graph labeled with numbers $0\leq p_v\leq 1$  such that $\sss(S)\geq 0$ for all $S\sbs V(G)$ and $\sss(\empt)>0$, then there is an assignment of numbers $0\leq x_v\leq 1$ ($v\in V(G)$) such that 
\begin{equation}
\phantom{\textrm{, and}}p_v=x_v\prod_{\substack{u\sim v\\ u\leq v}}(1-x_u)\textrm{, and}
\label{l.lprod}
\end{equation}
\begin{equation}
\prod_{v\in G}(1-x_v)=\sss(\empt).
\label{l.bbound}
\end{equation}
\end{theorem}
\noindent Thus, combined with Shearer's theorem, this tells us that for any chordal graph in $\LLL$, the Lefthanded Local Lemma can be used to prove its membership in $\LLL$---thus, for chordal graphs, Shearer's condition can be reduced to local conditions on the neighborhood of each vertex.  Note that in particular, Theorem \ref{t.lisbest} implies that for a given chordal graph, the choice of the order $\leq$ is irrelevant---if the Lefthanded Local Lemma applies to $G$ with one tree-order, it will apply with any tree-order satisfying Definition \ref{d.left}.  

\begin{ex}
  Returning to the example of $K_n^p$, the only tree order choice to make this a lefthanded graph is a linear order, say, $v_n\lneq v_{n-1}\lneq \cdots \lneq v_1$.   For this ordering, under the assignment $x_k=\frac{1}{k}$ for all $k$, the identity
\[
\frac{1}{k}\prod_{j=k+1}^n \frac{j-1}{j}=\frac{1}{n}
\]
implies that condition \eqref{l.cond} is satisfied (with equality in fact) for each $k$, although the product in the conclusion \eqref{l.lconc} is 0 since $x_1=1$.  Reassigning $x_1$ as $(1-\ep)$ for $\ep>0$ shows that the conclusion of the Lov\'asz Local Lemma holds for all $p=(1-\ep)\frac{1}{n}$, equivalently to Shearer's theorem.
\end{ex}

\smallskip

In Section \ref{s.alg} we will see that it can easily and efficiently be checked whether the Lefthanded Local Lemma applies to a given lefthanded graph (by showing that optimal assignments of the $x_v$ such as in the previous example can be found easily).  Together with Corollary \ref{c.G}, this implies that it can be efficiently checked whether any given labeled chordal graph is in the family $\LLL$.  Finally, in Section \ref{s.lop} we remark that the lopsided version of the Lefthanded Local Lemma characterizes the family of lopsidependency graphs to which the conclusion of the Local Lemma holds.

\section{Proof}
\label{s.proof}
Recall that $\leq$ is a tree order.  
We define
\[
D_v:=\{u \lneq v\},\;\;\bar D_v:= D_v\cup \{v\},
\]
and for any subset $U\sbs V$, we let
\[
\mu(U):=\{u\in U\st (\nexists w\in U) (u \lneq w)\}
\]
be the set of maximal elements of $U$.

For the sake of completeness, let's return to Observation \ref{o.leftissub}.

\noindent \textbf{Proving Observation \ref{o.leftissub}:}
  Suppose $G$ be a subtree graph where the underlying tree is $T$, and the vertices are subtrees $\tau_1,\tau_2,\dots,\tau_n\sbs T$.  Fix an arbitrary leaf $x_0$ of $T$, and let $P_i$ denote the unique shortest path from $\tau_i$ to $x_0$.  We define the relation $\leq_T$ on subtrees of $T$ by letting $\tau_i\leq_T\tau_j$ if either:
\begin{enumerate}
\item $P_i$ intersects $\tau_j$ and $P_j$ does not intersect $\tau_i$, \emph{or} 
\item $P_i$ intersects $\tau_j$, $P_j$ intersects $\tau_i$, and $i\leq j$.\label{en.arb}
\end{enumerate}
(Case \ref{en.arb} is ordering the pair arbitrarily according to the order $\tau_1,\dots,\tau_n$.)
It is not hard to check that $\leq$ is a tree-order, and that $G$ is a lefthanded graph with respect to $\leq$.

On the other hand, suppose $G$ is a lefthanded graph with respect to some tree-order $\leq$.   Define a tree $T$ on $V(G)$ where $v_1\sim v_2$ whenever $v_1\in \mu(D_{v_2})$ or $v_2\in \mu(D_{v_1})$.  To each $v\in G$, we associate the subtree $\tau_v$ of $T$ induced by the set 
\[
V(\tau_v):=\{u\leq v\st u=v\textrm{ or }u\sim v\}.
\]
  (These are subtrees by the definition of lefthanded graphs.)  It is not hard to check that $G$ is isomorphic to the intersection graph of the subtrees $\tau_v$.\qed

\bigskip
We will make use of a few more definitions.  We let 
\[
N_v:=\Gamma(v)\cap D_v=\{u\lneq v\st u\sim v\}, \textrm{ and}
\]
\[
F_v:=D_v\stm N_v.
\]

\noindent It will simplify our proof of Theorem \ref{t.lisbest} to get the following simple observation out of the way.  We include a detailed proof so that it is clear how our various definitions are being used.   $\dot\cup$ indicates a union which is always disjoint.
\begin{observ}
\label{o.key}
For any $v\in V$, we have 
\begin{equation*}
\mu(D_v)\dun \dot{\bigcup_{u\in N_v}}\mu(D_u)=N_v\dun \mu(F_v).
\label{l.key}
\end{equation*}
\end{observ}
\begin{proof}
We begin by proving containment in the $\sbs$ direction. 
If $w\in D_v\stm N_v$ is not maximal in $F_v$, then there is a $z\in N_v$ $w\lneq z$, but then $w$ isn't maximal in $D_v$ either.  This shows $\mu(D_v)\sbs N_v\dun \mu(F_v)$.
Next we consider the case where $w\in \mu(D_u)$ for some $u\in N_v$.  If $w\nin N_v$ then it is in $F_v$.  We must also have $w\in \mu(F_v)$ unless there is a vertex $z$ with $w\lneq z\lneq v$ such that $z\not\in N_v$.  Since $\leq$ is a tree-order $z$ and $u$ must be comparable.  But $w\lneq z\lneq u$ would contradict that $w\in \mu(D_u)$, and $w\lneq u\lneq z\lneq v$ would contradict condition (\ref{c.lordc}) in the definition of a lefthanded graph, since $u\in N_v$ but $z\not\in N_v$.

We now prove containment in the $\supset$ direction.  If $w\in \mu(F_v)$ is not a maximal element of $D_v$, then there is a vertex $u$ such that $w\lneq u\lneq v$, and we can choose $u$ to be a minimal such vertex with respect to $\leq$.  We have $u\in N_v$ since $w$ is maximal in $F_v$, and $w\in \mu(D_u)$ by our minimal choice of $u$.  Turning to the case where $w\in N_v$ instead, we have either that $w\in \mu(D_v)$, or else we can consider a minimal vertex $u$ satisfying $w\lneq u\lneq v$; we have $u\in N_v$ by condition (\ref{c.lordc}) in the definition of a lefthanded graph and we have $w\in \mu(D_u)$ by our minimal choice of $u$.
\end{proof}

Recall now that 
\[
\sss(S)=\sum_{\substack{I\supset S\\I\textrm{ indep.}}}(-1)^{\abs{I}-\abs{S}}\prod_{v\in I}p_v
\]
and define 
\begin{equation}
\BB(S):=\sum_{\substack{I\sbs S\\I\textrm{ indep.}}}(-1)^{\abs{I}}\prod_{v\in I}p_v.
\end{equation}

\begin{observ}
For any down-closed $S\sbs V$, we have
\[
\BB(S)=\prod_{w\in \mu(S)}\BB(\bar D_w).
\]
\label{o.Lasprod}
\end{observ}
\noindent (A set is `down-closed' if $D_s\sbs S$ for all $s\in S$. )
\begin{proof}
This follows from the definition of $\BB$, observing that since $\leq$ is a tree order we have that the sets $\bar D_w$ ($w\in \mu(S)$) form a partition of $S$, and have no edges between each other.
\end{proof}

\begin{observ}
If $(G,\leq)$ is a lefthanded dependency graph labeled with numbers $0\leq p_v\leq 1$ such that 
\begin{equation}
\sss(S)\geq 0\textrm{ for all }S\sbs V(G),
\label{l.sall}
\end{equation}
then  
\begin{equation}
0\leq \BB(S)\leq 1\textrm{ for all }S\sbs V(G).
\label{l.ls}
\end{equation}
Furthermore,
\begin{equation}
\BB(D_v)\geq p_v\BB(F_v).
\label{l.DgF}
\end{equation}

\end{observ}
\begin{proof}
It is easy to verify from the definitions (or c.f. \cite{shear}) that
\begin{equation}
  \sum_{R} \sss(R)=1
\label{l.ssum}
\end{equation}
and
\begin{equation}
 \BB(S)=\sum_{R\sbs \bar S}\sss(R).
\label{l.lassum}
\end{equation}
The latter sum is at least 0 by line (\ref{l.sall}), and at most 1 by line (\ref{l.ssum}).  This proves (\ref{l.ls}).  Line (\ref{l.DgF}) follows from line (\ref{l.ls}), coupled with the observation that 
\begin{equation}
\BB(D_v)-p_v\BB(F_v)=\BB(\bar D_v).
\end{equation}
\end{proof}
Finally, let us note that if $\BB(S)=0$ for any $S\sbs V$, then (\ref{l.sall}) and (\ref{l.lassum}) imply that $\sss(\empt)=\BB(V)=0$, and so Theorem \ref{t.lisbest} is trivially true.  Thus we may assume 
\begin{equation}
0<\BB(S)\textrm{ for all }S\sbs V(G).
\label{l.lpos}
\end{equation}

 We will show that Theorem \ref{t.lisbest} holds with the assignment
\begin{equation}
x_v=\frac{p_v\BB(F_v)}{\BB(D_v)},
\end{equation}
which is well-defined by line (\ref{l.lpos}).  Observe that lines (\ref{l.ls}) and (\ref{l.DgF}) imply that $0\leq x_v\leq 1$ for all $v$.  It can be checked directly that $x_v\geq p_v$ as well, but this is also a consequence of our proof.

Now for any non-minimal vertex we have that
\begin{multline}
x_v\prod_{u\in N_v}(1-x_u)=
\frac{p_v\BB(F_v)}{\BB(D_v)}  \prod_{u\in N_v}\left(1-\frac{p_u\BB(F_u)}{\BB(D_u)}\right)\\=
\frac{p_v\BB(F_v)}{\BB(D_v)}  \prod_{u\in N_v}\left(\frac{\BB(D_u)-p_u\BB(F_u)}{\BB(D_u)}\right)=
\frac{p_v\BB(F_v)}{\BB(D_v)}  \prod_{u\in N_v}\left(\frac{\BB(\bar D_u)}{\BB(D_u)}\right).
\label{l.spv}
\end{multline}

\noindent Observation 
  \ref{o.Lasprod} gives that 
\begin{equation}
\prod_{u\in N_v}\left(\frac{\BB(\bar D_u)}{\BB(D_u)}\right)=
\prod_{u\in N_v}\left(\BB(\bar D_u)\prod_{w\in \mu(D_u)}\frac{1}{\BB(\bar D_w)}\right),
\label{l.mout}
\end{equation}
and Observation \ref{o.key} implies that this product telescopes as
\begin{equation}
\prod_{u\in N_v}\left(\BB(\bar D_u)\prod_{w\in \mu(D_u)}\frac{1}{\BB(\bar D_w)}\right)=
  \left(\prod_{u\in \mu(D_v)}\BB(\bar D_u)\right)  \left(\prod_{u\in \mu(F_v)}\frac 1 {\BB(\bar D_u)}\right).
\label{l.usekey}
\end{equation}
Putting together lines (\ref{l.spv}), (\ref{l.mout}), and (\ref{l.usekey}), and then applying Observation \ref{o.Lasprod} to both products on the right-hand side of (\ref{l.usekey}), we get that
\begin{multline}
 x_v\prod_{u\in N_v}(1-x_u)=
\frac{p_v\BB(F_v)}{\BB(D_v)}\left(\prod_{u\in \mu(D_v)}\BB(\bar D_u)\right)  \left(\prod_{u\in \mu(F_v)}\frac 1 {\BB(\bar D_u)}\right)\\=
\frac{p_v\BB(F_v)}{\BB(D_v)}\cdot \frac{\BB(D_v)}{\BB(F_v)}=p_v
\end{multline}
and so  we have completed the proof of the first part (line (\ref{l.lprod})) of Theorem \ref{t.lisbest}.  To get the second part (line (\ref{l.bbound})), observe that with the assignment
\[
x_v=\frac{p_v\BB(F_v)}{\BB(D_v)},
\]
we have that 
\[
\prod_{v\in V}(1-x_v)=\prod_{v\in V}\left(1-\frac{p_v\BB(F_v)}{\BB(D_v)}\right)=\prod_{v\in V}\left(\frac{\BB(\bar D_v)}{\BB(D_v)}\right).
\]

Observation \ref{o.Lasprod} applied with $S=D_v$ for each $v$ implies that the final product telescopes to 
\[
\prod_{v\in V}\left(\frac{\BB(\bar D_v)}{\BB(D_v)}\right)=\prod_{v\in \mu(V)}\BB(\bar D_v)=\BB(V),
\]
and $\BB(V)=\sss(\empt)$ by the definitions of $\BB$ and $\sss$, giving
\[
\prod_{v\in V}(1-x_v)=\sss(\empt).
\]
This completes the proof of Theorem \ref{t.lisbest}.\qed

\section{Algorithm}
\label{s.alg}
To speak about the efficiency with which we can check whether a given graph is in $\LLL$, we must say something about the input structure of the labels $p_v$.  For our purposes we will just assume that the input structure allows efficient arithmetic and comparison operations (so the labels may be finite decimal expansions, or arbitrary rational numbers given as ratios of integers, \emph{etc.}).  Recall Corollary \ref{c.G}, which states that lefthanded orderings of chordal graphs can be found efficiently.  In this framework, the following is an algorithmic result:
\begin{cor}
Let $v_1,v_2,\dots,v_n$ be a linear extension of a lefthanded tree-order $\leq$ for some labeled chordal graph $G$.  Taking the vertices of $G$ in this order, and recursively computing
\begin{equation}
x_{v_i}=\bfrac{p_{v_i}}{\prod_{\substack{u\sim v_i\\u\leq v_i}}(1-x_u)}
\label{l.compute}
\end{equation}
for each $1\leq i\leq n$, we have that $G\in \LLL$ if and only if $0\leq x_{v_i}<1$ for each $x_{v_i}$.
\label{c.alg}
\end{cor}
\begin{proof}
Theorems \ref{t.ol} and \ref{t.lisbest} imply that to determine membership in $\LLL$, it is enough to determine whether there is an assignment $0\leq x_v<1$ satisfying all the conditions
\begin{equation}
p_v= x_v\prod_{\substack{u\sim v\\u\leq v}} (1-x_u).
\label{l.conds}
\end{equation}
If the assignments $x_v$ computed by the described procedure all lie in $[0,1)$, then they are such an assignment.  On the other hand, we see that each computation made according to line (\ref{l.compute}) is completely determined by the condition (\ref{l.conds}), so if any computed $x_v$ lies outside of $[0,1)$ then no satisfying assignment exists and Theorem \ref{t.lisbest} implies that $G\not\in \LLL$.
\end{proof}

\begin{ex}
  Consider the lefthanded realization of the Goldner--Harary graph shown in Figure \ref{f.GH}.  We will verify that the conclusion of the Local Lemma holds for Goldner--Harary graph labeled with $\frac 1 8$ at each vertex.  Beginning at the `bottom' of the tree order given in Figure \ref{f.GH}, we assign
\[
\begin{array}{rc}
x_a=x_d=x_c=x_i=x_h=x_k=&\frac 1 8,\\
x_b=x_j=\frac{1/8}{(7/8)^3}=&\frac {64}{343},\\
x_e=\frac{1/8}{(279/343)^2(7/8)^4}=&\frac{25088}{77841},\\
x_f=\frac{1/8}{(279/343)^2(7/8)^4(52753/77841)}=&\frac{25088}{52753},\\
x_g=\frac{1/8}{(279/343)^2(7/8)^4(52753/77841)(27665/52753)}=&\frac{25088}{27665},\\
\end{array}
\]
and since the computed assignments all lie in $[0,1)$, Corollary \ref{c.alg} implies that the conclusion of the Lov\'asz Local Lemma holds.  In the same manner one can verify that the conclusion of the Local Lemma holds with labels of $\frac{101}{800}$ at every vertex as well.  With labels of $\frac{102}{800}$, however, the computed $x_g$ is $\geq 1$,
and so Corollary \ref{c.alg} implies that the conclusion of the Local Lemma does not always hold for the Golder--Harary graph with this labeling.
\end{ex}

It should also be noted that when leaving $p$ as a variable, the algorithm described here gives a linear number of polynomial inequalities describing the values of $p$ for which uniform labeling by $p$ of a given graph implies membership in $\LLL$.    In the case of the Goldner--Harary graph, for example, the system can be solved by computer algebra software to find that the threshold value of $p$ is the smallest real root of the equation 
\[
1-11p+28p^2-29p^3+17p^4-6p^5+p^6=0,
\]
which lies between $.12689$ and $.126891$.

\section{Lopsidedness}
\label{s.lop}
In '91, Erd\H{o}s and Spencer proved a `Lopsided' version of the Lov\'asz Local Lemma, which relaxed the notion of independence required to apply the Local Lemma.  In their Lopsided Local Lemma, the notion of a dependency graph in which events are independent of families of non-neighbors is dropped.  They observed that the standard proof of the Local Lemma still works if we instead modify condition (\ref{l.lllcond}) to require that 
\[
  \pp\left(A_v|\bigcap_{w\in W}\bar A_w\right)\leq x_v\prod_{u\sim v}(1-x_u)
\]
for any family $W$ of nonneighbors of $v$ in the graph $G$.

Their lemma gives rise to a natural notion of a `lopsidependency graph' for a set of events $\{A_v\}_{v\in V}$, as a graph $G$ labeled with probabilities $0\leq p_v\leq 1$ such that for each $v$, 
\[
\pp\left(A_v|\bigcap_{w\in W}\bar A_w\right)\leq p_v
\]
for any family $W$ of nonneighbors of $v$.  Like the standard Local Lemma, the Lefthanded Local Lemma is true in a Lopsided sense:
\begin{theorem}[Lefthanded Lopsided Local Lemma \cite{leftseq}]
\label{t.lopleft}
Consider a family of events $\{A_v\}_{(v\in V)}$ with lefthanded lopsidependency graph $(G,\leq)$ with labels $p_v$.
If there is an assignment of numbers $0\leq x_v\leq 1$ such that
  \begin{equation}
   p_v\leq x_v\prod_{\substack{u\sim v\\ u\leq v}}(1-x_u),
    \label{l.lcond}
  \end{equation}
for all $v$, then we have
  \begin{equation}
    \pp\left(\bigcap_{A\in \aaa}\bar A\right)\geq \prod_{v\in G}(1-x_v).
\label{l.llconc}
  \end{equation}
\end{theorem}

The notion of a lopsidependency graph allows us to define a family $\LLL^L$ as the family of graphs $G$ with vertices labeled with real numbers $0\leq p_v\leq 1$ with the property that any family of events $A_v$ having $G$ as a lopsidependency graph satisfies $\pp(\bigcap \bar A_v)>0$.  One question that seems natural is whether Theorem \ref{t.lopleft} charaterizes the $\LLL^L$ for chordal graphs, as Theorem \ref{t.ol} does for $\LLL$.  This is a simple consequence of Theorem \ref{t.lisbest} and the relation $\LLL^L\sbs \LLL$, however: if a graph $G$ is in $\LLL^L$ then it is in $\LLL$ as well, so applying Theorem \ref{t.lisbest} gives an assignment $\{x_v\}$ with which Theorem \ref{t.lopleft} applies.

We close by noting that Scott and Sokal\cite[cf.~Thm 3.1]{gas} pointed out that Shearer's characterization applies to $\LLL^L$ as well, and so  $\LLL^L=\LLL$ in fact holds.  Via this observation, Theorem \ref{t.lopleft} and the Lopsided Local Lemma of Erd\H{o}s and Spencer can actually be viewed as consequences of their non-lopsided versions.

\end{document}